\documentclass{article}
\usepackage[utf8]{inputenc}
\usepackage{hyperref}
\usepackage{colortbl}
\usepackage[margin=1.25in]{geometry}
\usepackage[normalem]{ulem}
\usepackage{lscape}
\usepackage{wrapfig}
\usepackage{booktabs}
\usepackage{bbm}
\usepackage{xcolor}
\usepackage{graphicx}
\usepackage{amssymb}
\usepackage{amsbsy}
\usepackage{amsmath}
\usepackage{amsthm}
\usepackage{amsfonts}
\usepackage{enumitem}
\usepackage{subcaption}
\usepackage{caption}
\usepackage{listings}
\usepackage{authblk}
\usepackage[]{natbib}   
\usepackage{algorithm}
\usepackage{tabularx}
\usepackage{algpseudocode}
\usepackage{accents}
\usepackage{multicol}
\usepackage{lipsum} 
\usepackage[normalem]{ulem}

\makeatletter
\newcommand{\multiline}[1]{%
  \begin{tabularx}{\dimexpr\linewidth-\ALG@thistlm}[t]{@{}X@{}}
    #1
  \end{tabularx}
}
\makeatother

\newcommand{\ignore}[1]{}

\hypersetup{
    colorlinks=true,
    linkcolor=blue,
    filecolor=magenta,      
    urlcolor=cyan,
}
\newtheorem{example}{Example}
\newtheorem{proposition}{Proposition}

\newtheorem{corollary}{Corollary}

\newcommand{\la}[0]{\leftarrow}

\newcommand{\br}[1]{\{ {#1} \}}
\newcommand{\mc}[1]{\mathcal{#1}}
\newcommand{\om}[0]{\omega}
\newcommand{\ti}[1]{\tilde{#1}}

\newcommand{\pr}[1]{\{ #1 \}}

\newcommand{\tcb}[1]{{#1}}

\title{A computational study of cutting-plane methods for multi-stage stochastic integer programs}
\date{\today}
\author{Akul Bansal}
\author{Simge K\"u\c{c}\"ukyavuz}
\affil{Industrial Engineering and Management Sciences \\ Northwestern University \\
{\texttt{akul@u.northwestern.edu, simge@northwestern.edu}}}

\begin{document}

\maketitle

\begin{abstract}
We report a computational study of cutting plane algorithms for multi-stage stochastic mixed-integer programming models with the following cuts: (i) Benders', (ii) Integer L-shaped, and (iii) Lagrangian cuts. We first show that Integer L-shaped cuts correspond to one of the optimal solutions of the Lagrangian dual problem, and, therefore, belong to the class of Lagrangian cuts. To efficiently generate these cuts, we present an enhancement strategy to reduce time-consuming exact evaluations of integer subproblems by alternating between cuts derived from the relaxed and exact computation. Exact evaluations are only employed when Benders' cut from the relaxation fails to cut off the incumbent solution. Our preliminary computational results show the merit of this approach on multiple classes of real-world problems.
\end{abstract}

\section{Introduction}
\tcb{We consider a multi-stage stochastic integer programming (MSIP) model which is a framework for sequential decision-making under uncertainty when decision variables have integer restrictions.} Specifically, MSIP models the stochastic version of the following deterministic problem with $T$ stages:
\begin{equation*}
\min _{\left(x_1, y_1\right), \ldots,\left(x_T, y_T\right)}\left\{\sum_{t=1}^T f_t\left(x_t, y_t\right):\left(x_{t-1}, x_t, y_t\right) \in X_t, \forall t=1, \ldots, T\right\} .
\end{equation*}
The formulation has a special structure where the $h_t$-dimensional state variable vector $x_t \in \mathbb{R}^{h_t}$ links successive stages, and the local \tcb{variable vector} $y_t$ is only part of the subproblem at \tcb{stage $t = 1, \ldots, T$. The objective function in stage $t$ is denoted by $f_t$ and the set of constraints is denoted by $X_t$, which include integrality restrictions on some or all variables}. The data required in stage $t$ for the objective function $f_t$ and constraints in $X_t$ are denoted as $\omega_t := (f_t, X_t)$.
When data $(\omega_2, \ldots, \omega_T)$ are uncertain it is viewed as a stochastic process, denoted as $(\ti{\omega}_2, \ldots, \ti{\omega}_T)$. A practical approach in this case is to model the stochastic process with a finite number of realizations in the form of a scenario tree. Let $\mathcal{T}$ be the scenario tree we associate with the underlying stochastic process, and let $S_t$ denote the set of nodes in stage $t$. Each node $n$ in stage $t$ represents a unique realization $\omega_t^n$ of the random variable $\ti{\omega}_t$. The parent node in stage $t-1$ is denoted by $a(n)$ and children nodes of node $n$ are denoted by $\mathcal{C}(n)$. The conditional probability of transitioning from node $n$ to node $m$ is given as $q_{nm}$. \tcb{The MSIP can then be formulated using the following dynamic programming (DP) recursions}:
\begin{equation} \label{prob:msp}
\min _{x_1, y_1}\left\{f_1\left(x_1, y_1\right)+\sum_{m \in \mathcal{C}(1)} q_{1 m} Q_m\left(x_1, \om_2^m \right):\left(x_{a(1)}, x_1, y_1\right) \in X_1\right\},
\end{equation}
where for each node $n \in S_t$ and $t \in \mathcal{T} \backslash \{1\}$,
\begin{equation} \label{prob:Qval}
Q_n\left(x_{a(n)}, {\om}_t^n \right)=\min _{x_n, y_n}\left\{f_n\left(x_n, y_n\right)+\sum_{m \in \mathcal{C}(n)} q_{n m} Q_m\left(x_n, {\om}_t^m\right):\left(x_{a(n)}, x_n, y_n\right) \in X_n\right\}.
\end{equation}
For a given $x_{a(n)}$, the function $Q_n(\cdot)$ is often referred as \tcb{the} optimal value function at node $n$, and $F_n(\cdot) := \sum_{m \in \mathcal{C}(n)} q_{n m} Q_m\left(\cdot\right)$ is referred to as the expected cost-to-go function. \tcb{Note that since node $n$ corresponds to unique scenario realization $\om_t^n$, the dependence on $\om_t^n$ in the right hand side of \eqref{prob:Qval} is captured only through subscript $n$.} In our computational study, we focus on MSIP problems where the objective function $f_n$ is linear, and the constraint system $X_n$ has linear constraints of the form $B_n x_{a(n)} + A_n x_n + C_n y_n \geq b_n$ along with integrality restrictions on a subset of the variables. \tcb{These assumptions typically make the optimal value function $Q_n(\cdot)$ non-linear and non-convex (\citealt{blair1982value})}.

The focus of this paper is on designing efficient decomposition methods for solving MSIP problems. We begin with discussing prior work in this direction. 
The multi-stage stochastic program \eqref{prob:msp} can be solved as \tcb{a large-scale} integer program using the extensive form
\begin{equation} \label{eq:ext_form}
\min _{x_n, y_n}\left\{\sum_{n \in \mathcal{T}} p_n f_n\left(x_n, y_n\right):\left(x_{a(n)}, x_n, y_n\right) \in X_n, \forall n \in \mathcal{T}\right\},
\end{equation}
where $p_n$ is the probability of realization of the unique data sequence from the root node to \tcb{node $n \in S_t$ for some $t \in \mathcal{T}$}. While the above formulation makes the problem deterministic, the size of the problem becomes exponential as the size of the scenario tree grows with stages and scenarios. {This approach is, therefore, not viable for stochastic programs with large scenario trees.} \tcb{To circumvent this difficulty, problem \eqref{prob:msp} is addressed using decomposition algorithms in practice.} {One way to decompose the problem is to do so at the level of individual nodes of the scenario tree. For a node $n$, this is achieved by approximating the future cost $F_n (\cdot)$ with variable $\theta_n$ and cuts of the form $\theta_n \geq h(x_n)$, where $h$ is usually an affine function of the state vector $x_n$.} The cuts provide a lower bound on the future cost $F_n$ and are updated iteratively by traversing the entire scenario tree in the Nested Benders' algorithm (\citealt{birge1985decomposition}). While this algorithm allows the decoupling of the extensive form into subproblems of smaller size, the number of subproblems solved in each iteration depends on the size of the scenario tree. A computationally more practical approach is to use stochastic dual dynamic integer programming (SDDiP) which improves the approximation of the future cost along a subset of scenario paths sampled in each iteration (\citealt{pereira1991multi, zou2019stochastic}). Specifically, in iteration $i$ of the algorithm the approximation for the subproblem with value function $Q_n(\cdot)$ is given by:
\begin{align} \label{prob:node_sub}
\left(P_n^i\left(x_{a(n)}^i, \psi_n^i, \om_t^n \right)\right): \underline{Q}_n^i\left(x_{a(n)}^i, \psi_n^i, \om_t^n \right):= \min _{x_n, y_n, z_n} &f_n\left(x_n, y_n\right)+\psi_n^i\left(x_n\right)  \\
\text { s.t.}& \left(z_n, x_n, y_n\right) \in X_n, \nonumber \\
& z_n = x_{a(n)}^i, \label{sub:copy}
\end{align}
where the approximate expected cost-to-go function is determined as:
\begin{align} \label{cut:general}
\psi_n^i\left(x_n\right):=\min \left\{ \theta_n: \theta_n \geq \sum_{m \in \mathcal{C}(n)} q_{n m} \cdot\left(v_m^{\ell}+\left(\pi_m^{\ell}\right)^{\top} x_n\right), \forall \ell=1, \ldots, i-1 \right\}.
\end{align}
The notation $\left(P_n^i\left(x_{a(n)}^i, \psi_n^i, \om_t^n \right)\right)$ denotes the subproblem at node $n$ in iteration $i$, with incumbent solution $x_{a(n)}^i$ and lower bound on expected cost-to-go function given by $\psi_n^i$. As before, the dependence on the unique scenario realization $\om_t^n$ is captured only through subscript $n$. Variable $z_n$ captures the incumbent state solution from the previous stage. The approximation $\psi_n^i (\cdot)$ in \eqref{cut:general} can more generally be a non-linear function such as the step function (\citealt{philpott2020midas}). The objective function of the subproblem is denoted as $\underline{Q}_n^i\left(x_{a(n)}^i, \psi_n^i, \om_t^n \right)$ which is used to determine the approximate expected cost-to-go function $\underline{F}_n^i (x_n^i) := \sum_{m \in \mathcal{C}(n)} q_{n m} \underline{Q}_m\left(x_n^i, \psi_m^i, \om_t^m \right)$. The initial approximation $\psi_n^1$ is taken to be a trivial lower bound on $F_n(\cdot)$. In iteration $i$, the lower bound $\psi_n^i$ can be improved by evaluating $\underline{F}_n^{i}(x_n^i)$ exactly as an integer program or by solving its relaxation. If $\psi_n^i$ is less than the evaluation, a cut of the form \eqref{cut:general} is added.

\tcb{A breadth of work in solving MSIP problems focuses on deriving cuts that provide a convex relaxation of cost-to-go functions. For example, Benders' cuts (\citealt{benders1962partitioning}) are obtained from linear programming (LP) relaxation. They are sufficient to provide convergence for solving MSIP problems with continuous recourse in every stage, i.e., when decision variables from the second stage onwards $\{x_t, y_t\}_{t \geq 2}$ are real-valued. Several methods based on convexifying the subproblems have been proposed in context of two-stage stochastic programs. These include disjunctive programming used in \citealt{sen2005c, sen2006decomposition, qi2017ancestral}, Fenchel cuts (\citealt{ntaimo2013fenchel}), and scaled cuts by \citealt{van2020converging} which integrate previously generated Lagrangian cuts into a cut-generating non-linear program. \citealt{gade2014decomposition} and  \citealt{zhang2014finitely} utilize Gomory cuts to obtain iteratively tighter approximations of the pure-integer second-stage subproblems. These methods do not tend to scale well for multi-stage problems. We refer the reader to \citealt{KS2017} for a review of methods based on convexifications of second-stage problems.

When the recourse variables are integer-restricted, Benders' cuts tend to be weak and the iterative procedure only converges to a lower bound. Several papers, therefore, derive cuts based on Lagrangian relaxation which provides better bounds than Benders' cuts (\citealt{gjelsvik1999algorithm, thome2013non, steeger2015strategic}). Further, in the case of binary state variables, convergence can be ensured with Integer L-shaped cuts (\citealt{laporte1993integer}) or Lagrangian cuts (\citealt{caroe1999dual, zou2019stochastic}). These cuts are derived from the exact evaluation of integer subproblems and are computationally expensive to compute than the Benders' cuts. The current literature aims to combine the two kinds of cuts to take advantage of their computational complexity and convergence properties in accelerating the decomposition algorithms.  \citealt{rahmaniani2020benders}, for instance, discuss a three-phase strategy for generating cuts in their Nested Benders implementation for two-stage problems. {They propose strengthened Benders' cuts (to be defined in Section \ref{sec:cut_fam}) which are preceded by Benders' cuts in the first phase and then followed by the Lagrangian cuts in the third phase.}  With each phase, the strength of the cut and the time taken to derive the cut increases. Applying Benders' cuts initially is a computationally less-intensive approach to eliminate incumbent solutions that are less likely to be optimal. \citealt{zou2019stochastic} perform experiments in which more than one cut is added at a node in a single iteration. For instance, one of the combinations they consider adds a strengthened Benders’ cut, an Integer L-shaped, and a Lagrangian cut at the same time. With this approach, the authors report improvement in the solution time of the algorithm for one class of problems they tested. Note that deriving more than one cut in each iteration may improve the bound obtained through $\psi_n^i$, but also increases the computational burden. It is not clear from the experiments in \citealt{zou2019stochastic} if combined cuts are advantageous for other classes of problems than the one considered. For two-stage stochastic integer programs, \citealt{angulo2016improving} propose to avoid the time-consuming exact evaluation of the second-stage subproblems by alternating between a Benders cut derived from the linear relaxation and an Integer L-shaped cut derived from the integer subproblem. Specifically, cuts based on exact evaluations are only added when the Benders' cut obtained from the relaxation fails to cut off the incumbent solution. This approach is different from \citealt{rahmaniani2020benders} which utilizes only one kind of cut in each phase. This means that Benders cuts are not utilized in later stages while they could have helped cut off incumbent solutions at a computationally cheaper price.} 

\tcb{\citealt{chen2022generating} investigate new techniques for generating Lagrangian cuts in a two-stage setting. Instead of solving the Lagrangian dual problem over all Lagrange multipliers, they propose to solve it over a restricted space. The restricted space is taken to be the span of coefficients of Benders cuts employed in the last few iterations and identified through a mixed-integer programming (MIP) approximation. The key idea is that the proposed separation problem should require fewer evaluations of the exact subproblem. {This is further reinforced by solving the separation problem sub-optimally with higher gap tolerance}. The authors show that their proposed approach helps close the gap at the root node of the master problem. {In our computational experiments (Appendix \ref{sec:chen}), we discuss how the authors' cut-generating method for Lagrangian cuts compares with our proposed approach.}} 

\tcb{To deal with the computational challenge of solving MSIP problems, this paper focuses on enhancing the cut-generating procedure for decomposition methods. Specifically, we aim to extend the alternating cut procedure \citealt{angulo2016improving} to the multi-stage setting. The main contributions of our work are summarized below:
\begin{enumerate}
    \item We present a variant of the SDDiP algorithm that utilizes the alternating cut procedure to solve MSIP problems. While the original paper \citealt{angulo2016improving} focuses on alternating between Benders' cuts and Integer L-shaped cuts, we consider alternating between any two cuts. {Specifically, we alternate between a valid cut, which accelerates the cut-generating procedure, and a tight cut, which ensures convergence.}
    \item {We show that Integer L-shaped cuts belong to the family of Lagrangian cuts. The Lagrangian dual problem obtained by relaxing the non-anticipativity constraints \eqref{sub:copy} might admit multiple optimal solutions. Integer L-shaped cuts correspond to one of the optimal solutions. Further, we discuss how different optimal solutions result in cuts of varying strengths.}
    
    \item We conduct an extensive computational experiment on two classes of MSIP problems.  {We compare how the alternating cut procedure performs with different choices of cuts, namely, the Integer L-shaped cut and the Lagrangian cut obtained by solving the dual problem using a subgradient algorithm.} We find that this approach has the potential to outperform the default approach of adding a single cut {from the same family} at each node in each iteration. 
\end{enumerate}}

The remainder of this paper is organized as follows. In Section \ref{sec:sddip}, we review the stochastic dual dynamic integer programming algorithm (SDDiP). We discuss various families of cuts and prove that Integer L-shaped cuts belong to the class of Lagrangian cuts. In Section \ref{sec:alt}, we discuss the proposed approach for efficient cut generation. Our computational experiments are presented in Section \ref{sec:comp}. Finally, we conclude in Section \ref{sec:concl}.

\section{Stochastic dual dynamic integer programming} \label{sec:sddip}

We now review the SDDiP algorithm proposed in \citealt{zou2019stochastic} under two assumptions. First, we assume stochasticity satisfies stage-wise independence. This means that children nodes in stage $t + 1$ are the same, and have the same conditional probabilities irrespective of the parent node in stage $t$. Owing to this, we can maintain a single value function $Q_t (\cdot)$ per stage instead of having value functions $Q_n(\cdot)$ at each node $n \in \mathcal{T} \backslash \{1\}$. Second, we assume that the MSIP problem has relatively complete recourse, which means for any incoming state solution from the previous stage, the subproblem in the current stage is feasible. This assumption can always be achieved by adding auxiliary continuous variables and penalizing them in the objective function.

To describe the algorithm with stage-wise independence we re-define some of the notation. {We adopt the same notation as in \citealt{zou2019stochastic}.} Let $N_t$ be the number of realizations of uncertain parameters at stage $t$. Irrespective of the state of the system in stage $t-1$, the possible uncertainty realizations in stage $t$ remain the same. We, therefore, denote the uncertainty realization in stage $t$ with index $j$ as $\om^j_t$.  The notation for the conditional probability of transitioning from any node in stage $t-1$ to a node $j \in \{1, \ldots, N_t\}$ in stage $t$, is simplified to $q_t^j$. The total number of scenarios is $N = \prod_{t = 1}^{T} N_t$. The value function $Q_t(\cdot)$ is parameterized by the solution $x_{t-1}$ from the previous stage and the uncertainty realization $\om^j_t$ in stage $t$. The approximate expected cost-to-go function in stage $t$ and before iteration $i$ is denoted as $\psi_{t}^{i}$. The SDDiP algorithm refines the value function along a subset of scenario paths sampled from the scenario tree. Let $k$ denote one of the $M$ scenario paths sampled. With a slight abuse of notation, we also use $\omega_{t}^{k}$ to denote the uncertainty realization in stage $t$ and the sample $k \in \br{1, \ldots, M}$. {Further, let $P_t^i\left(x_{t-1}^{ik}, \psi_t^i, \omega_t^k\right)$ denote the forward problem in iteration $i$ and at stage $t$ given the state solution $x_{t-1}^{ik}$ from the previous stage, approximation function $\psi_t^i$ of the future cost, and $\omega_t^k$ as the uncertainty realization. The definition of subproblem $P_t^i\left(x_{t-1}^{ik}, \psi_t^i, \omega_t^k\right)$ remains same as \eqref{prob:node_sub} with subscript $n$ replaced with $t$ and $w_t^n$ replaced with $w_t^k$.} In the backward step, a suitable relaxation $R_t^{ij}$ of the updated problem $P_t^i\left(x_{t-1}^{ik}, \psi_t^{i+1}, \omega_t^{j}\right)$ for some $1 \leq j \leq N_t$, is solved to obtain the cut coefficients $v_{t}^{ij}, \pi_{t}^{ij}$.
The cut coefficients for each uncertainty realization $\omega_t^{j}$, are then aggregated as follows to obtain the final cut:
\begin{align} \label{cut:form}
    \theta_{t-1} \geq \sum_{j=1}^{N_t} q_t^j \left(v^{ij}_{t} + (\pi^{ij}_{t})^{\top} x_{t-1} \right).
\end{align}
The SDDiP algorithm is described in Algorithm \ref{algo:sddip}. In each iteration $i$ of the algorithm, we perform three key steps: the sampling step (Line 3), the forward step (Lines 4-10), and the backward step (Lines 14-22). In the sampling step, we sample $M$ scenarios with replacement, that is $M$ paths from the root to the leaf node of the scenario tree. Next, in the forward step, for each scenario $k = 1, \ldots, M$ we traverse through the path $\pr{\om^{k}_{t}}_t$ by solving the approximate subproblem $P_{t}^{i}(x_{t-1}^{ik}, \psi^{i}_{t}, \omega_{t}^{k})$. The subproblem is parameterized by the incoming state variable $x_{t-1}^{ik}$ from the previous stage, current approximation $\psi^{i}_{t}$ of the value function, and the uncertainty realization $\omega_{t}^{k}$. The solution of the state variable $x_t^{ik}$ is then passed to the next stage problem. An upper bound (UB) cannot be obtained in the SDDiP algorithm because we sample only a subset of scenario paths. The upper bound update is therefore an estimate taken to be the right end of the confidence interval based on $M$ sampled paths. This is computed in Lines 11-13. When all sampled scenario paths have been traversed in the forward path, an update of the approximate cost-to-go function $\psi_{t}^{i}$ is obtained in the backward step. Unlike the forward step, the update is attained by solving all $N_{t+1}$ updated subproblems at stage $t+1$, and not just the sampled ones. Finally, a lower bound (LB) is computed in Lines 23-25, by solving the problem in the first stage.

The SDDiP algorithm differs from the more classical Nested Benders (NB) algorithm in three ways. {First, NB traverses all scenario paths instead of considering only a small subset of $M$ paths.} Owing to this, the second difference arises in the upper bounds attained in NB which are deterministic and not an estimate from a random sample. Third, the convergence criterion in NB is taken to be the gap between the lower and upper bound attained. This is not possible in SDDiP because the upper bounds attained here are only estimates. \cite{zou2019stochastic}[Theorem 1] proves almost sure convergence of the SDDiP algorithm to the optimal solution of MSIP. For finite termination of the SDDiP algorithm we discuss a suitable convergence criterion in Section \ref{sec:cc}. Next, we review various types of cuts that are used in the SDDiP algorithm.

\begin{algorithm}[t]
	\caption{Stochastic Dual Dynamic Integer Programming} 
    \label{algo:sddip}
	\begin{algorithmic}[1]
        \State Initialization: $i \leftarrow 1$, $LB \la - \infty$, $UB \la + \infty$, lower bounds $\pr{L_t}$, initial approximation $\pr{\psi_{t}^{1}(\cdot) = \pr{\theta_t: \theta_t \geq L_t}}$ and starting state variable $x_0$.
        
		\While {\text{stopping criterion is not satisfied}   }
                \State Sample $M$ scenarios $\Omega^{i} = \pr{\om_{1}^{k}, \ldots, \om_{T}^{k}}_{k = 1, \ldots, M}$

                \State /* \textbf{Forward step} */
                \For {$k = 1 \ldots M$}
                    \For {$t = 1 \ldots T$}
                        \State \multiline{solve forward problem $P_{t}^{i} (x^{ik}_{t-1}, \psi^{i}_{t}, \om_{t}^k)$ and collect solution $(x^{ik}_{t}, y^{ik}_{t}, z_{t}^{ik}, \theta_{t}^{ik} = \psi_{t}^{i}(x_t^{ik}))$.}
                    \EndFor  
                    \State $u^{k} \la \sum_{t = 1, \ldots, T} f_t(x_t^{ik}, y_t^{ik}, \om_t^k)$
                \EndFor
                \item[]
                \State /*  \textbf{Statistical upper bound update} */
                \State $\hat{\mu} \la \frac{1}{M} \sum_{k=1}^{M} u^{k}$ and ${\hat{\sigma}}^2 \la \frac{1}{M-1} \sum_{k = 1}^{M} (u^k - \hat{\mu})^2$
                \State $UB \la \hat{\mu} + z_{\alpha/2} \frac{\hat{\sigma}}{\sqrt{M}}$.
                \item[]
                \State /* \textbf{Backward Step} */
                \For {$t = T, \ldots, 2$}
                    \For{$k = 1, \ldots, M$}
                        \For {$j = 1, \ldots, N_t$}
                                \State \multiline{%
                                                Solve a suitable relaxation $R_t^{ij}$ of the updated problem $P_{t}^{i}(x_{t-1}^{ik}, \psi_t^{i+1}, \om_t^{j})$ and collect coefficients $(v_t^{ij}, \pi_{t}^{ij})$.}
                        \EndFor
                        \State \multiline{Add optimality cut using the coefficients $\br{(v_t^{ij}, \pi_t^{ij})}_{j = 1, \ldots, N_t}$ to $\psi_{t-1}^i$ to get $\psi_{t-1}^{i+1}$.}
                \EndFor
            \EndFor
            \item[]
            
            \State /*  \textbf{Lower bound update} */
            \State \multiline{Obtain the lower bound LB by solving $P^{i}({x}_0, \psi_{1}^{i+1})$.}
            \State $i \la i + 1$
		\EndWhile
	\end{algorithmic} 
\end{algorithm}

\subsection{Cut families} \label{sec:cut_fam}
The central theme of decomposition algorithms is to decompose the problem into subproblems at each node by using approximations for the expected cost-to-go functions. Relaxations of individual subproblems are then solved to update these approximations using linear hyperplanes or cuts. The SDDiP algorithm converges if the cuts are valid, tight, and finite. \tcb{The cut $\theta_t \geq v^{ij}_{t} + (\pi^{ij}_{t})^{\top} x_{t-1}$ obtained in the backward step by solving the relaxation $R_t^{ij}$ of the subproblem $P_t^{i}(x_{t-1}^{ik}, \psi_{t}^{i}, \omega_{t}^{j})$, is said to be \emph{tight} if the cut} coincides with the updated approximate value function $\underline{Q}_t^i\left(\cdot, \psi_t^{i+1}, \om_{t}^{j} \right)$ at the forward step solution $x_{t-1}^{ik}$, that is,
\begin{equation*} 
    \underline{Q}_t^i\left(x_{t-1}^{ik}, \psi_t^{i+1}, \om_{t}^{j} \right) = v^{ij}_{t} + (\pi^{ij}_{t})^{\top} x_{t-1}^{ik}.
\end{equation*}

Note that the tightness property does not require the cuts to be tight for the true value function $Q_t(x_{t-1}, \om_{t}^{j})$, but tight for the approximate value function $\underline{Q}_t^i\left(x_{t-1}^{ik}, \psi_t^{i+1}, \om_{t}^{j} \right)$. \tcb{When the linear approximation $v^{ij}_{t} + (\pi^{ij}_{t})^{\top} x_{t-1}^{ik}$ only provides a lower bound on the updated value function $\underline{Q}_t^i\left(\cdot, \psi_t^{i+1}, \om_{t}^{j} \right)$, the cut is only considered \emph{valid}. Specifically, in this case we have:
\begin{equation*} 
    \underline{Q}_t^i\left(x_{t-1}^{ik}, \psi_t^{i+1}, \om_{t}^{j} \right) \geq v^{ij}_{t} + (\pi^{ij}_{t})^{\top} x_{t-1}^{ik}.
\end{equation*}
The \emph{finiteness} property requires that in each iteration $i$, the backward step can only generate finitely many distinct cut coefficients $(v^{ij}_{t}, \pi^{ij}_{t})$.} We now briefly review four classes of cuts that are commonly used. 

\textbf{Benders' cut:} The first collection of cuts used in stochastic programming are Benders' cut (\citealt{benders1962partitioning}). For the subproblem at stage $t-1$, the cuts are obtained from an optimal dual solution of the LP relaxation of updated subproblems $P_t^{i}(x_{t-1}^{ik}, \psi_{t}^{i+1}, \omega_{t}^{j})$ for $1 \leq j \leq N_t$. Specifically, let $Q^{i, LP}_{t} (x_{t-1}^{ik}, \psi_{t}^{i+1}, \omega_{t}^{j})$ denote the optimal objective value of the LP relaxation of $P_t^{i}(x_{t-1}^{ik}, \psi_{t}^{i+1}, \omega_{t}^{j})$. Further, let $\pi_{t}^{ij}$ denote the dual solution corresponding to constraints $z_t = x_{t-1}^{ik}$. Benders' cut is then given as follows:
\begin{align*}
    \theta_{t-1} \geq \sum_{j = 1}^{N_t} q_t^j Q^{i,LP}_{t} (x_{t-1}^{ik}, \psi_{t}^{i+1}, \omega_{t}^{j}) + \sum_{j = 1}^{N_t} q_{t}^{j} (\pi_j^i)^T (x_{t-1} - x_{t-1}^{ik}).
\end{align*}
For subproblems with integer restrictions, Benders' cut is only a valid (not necessarily tight) cut.

\textbf{Integer L-shaped cut:} \citealt{laporte1993integer} propose integer L-shaped cuts when the state variables are binary. Let $x_{tr}$ denote the $r^{th}$ component of the variable vector $x_t$.  Further, suppose that we know a lower bound $L(\psi_{t}^{i+1}, \om_t^j)$ on the subproblem $\underline{Q}_t^i(\cdot, \psi_{t}^{i+1}, \om_t^j)$ in stage $t$. Given incumbent solution $x_{t-1}^{ik}$ in iteration $i$ and along the sampled path $k$, the integer L-shaped cut is given as follows:

\begin{align}
    \theta_{t-1} \geq  \sum_{j} q_{t}^{j} \Biggr(
    &\left( \underline{Q}_t^i(x_{t-1}^{ik}, \psi_{t}^{i+1}, \om_t^j) - L(\psi_{t}^{i+1}, \om_t^j)\right) \left(\sum_{r: x_{t-1,r}^{ik} = 1} \left(x_{t-1,r} - 1\right) - \sum_{r: x_{t-1,r}^{ik} = 0} x_{t-1,r} \right) \nonumber\\
    &+ \underline{Q}_t^i(x_{t-1}^{ik}, \psi_{t}^{i+1}, \om_t^j) \Biggr). \label{cut:intL}
\end{align}
Note that when the variable vector $x_{t-1}$ is component-wise the same as $x_{t-1}^{ik}$, then we get 
$$\theta_{t-1} \geq \sum_{j=1}^{N_t} q_j \underline{Q}_t^i(x_{t-1}^{ik}, \psi_{t}^{i+1}, \om_t^j).$$ 
This allows to improve the approximation of the cost $\theta_{t-1}$ as $\sum_{j=1}^{N_t} q_j \underline{Q}_t^i(x_{t-1}^{ik}, \psi_{t}^{i+1}, \om_t^j)$ over the trivial lower bound $L(\psi_{t}^{i+1}, \om_t^j)$. However, when $x_{t-1}$ is not the same as $x_{t-1}^{ik}$, then the inequality \eqref{cut:intL} only provides a trivial lower bound on the future cost. For this reason, the derived cut tends to be tight at the current incumbent solution but is loose at other solutions. Unlike the Benders' cut, they are obtained by solving the subproblem exactly; hence they are valid and tight.

\textbf{Lagrangian cuts:} To obtain the Lagrangian cuts, consider the subproblem $P_t^{i}(x_{t-1}^{ik}, \psi_{t}^{i+1}, \om_t^{j})$:
\begin{subequations} \label{prob:sub+Copy_old}
\begin{align}
\underline{Q}_t^i(x_{t-1}^{ik}, \psi_{t}^{i+1}, \om_t^j) = \min_{x_t, y_t} \hspace{0.1cm} &f_t\left(x_t, y_t, \om_t^{j} \right)+\psi_t^{i+1}\left( x_t \right) \nonumber \\
\text { s.t. } &\hspace{0.1cm}z_t = x_{t-1}^{ik}, \label{cons:copy} \\
&\hspace{0.1cm}z_t \in \br{0,1}^{h_t} \label{cons:binary_restr}\\
&(z_t, x_t, y_t) \in X_t. \nonumber
\end{align}
\end{subequations}
The constraint \eqref{cons:binary_restr} is redundant in presence of non-anticipativity constraint \eqref{cons:copy}, but plays am important role in Proposition \ref{prop:intL}.
The Lagrangian cut is obtained by solving the Lagrangian dual of the reformulated problem where the non-anticipativity constraint is relaxed. Specifically, in the backward step in iteration $i$ at stage $t$, we solve the following problem:
\begin{align} \label{prob:ldual}
R_t^{ij} := \max_{\pi_{t}^{ij}} \left( \mc{L}_t^{i}(\pi_{t}^{ij}; \om_t^{j}) + (\pi_{t}^{ij})^{\top} x_{t-1}^{ik}  \right),
\end{align}
where the Lagrangian function $\mc{L}_t^{i}(\pi_{t}^{ij}; \om_t^{j})$ is given as follows:
\begin{subequations}\label{prob:lag_function}
\begin{align} 
\mc{L}_t^{i}(\pi_{t}^{ij}; \om_t^{j}) = \min_{x_t, y_t, z_t} \hspace{0.1cm} &f_t\left(x_t, y_t, \om_t^{j} \right)+\psi_t^{i+1}\left( x_t \right) - (\pi_{t}^{ij})^{\top} z_t \label{lagrn_obj_old}\\
\text { s.t. } &z_t \in \br{0,1}^h\\
&(z_t, x_t, y_t) \in X_t.
\end{align}
\end{subequations}

Clearly, for any value of $\pr{\pi_{t}^{ij}}_{j}$ the cut given by 
\begin{equation}\label{cuts:lagn}
    \theta_{t-1} \geq \sum_{j = 1}^{N_t} q_t^j \left( \mc{L}_t^{i}(\pi_{t}^{ij}; \om_t^{j}) + (\pi_{t}^{ij})^{\top} x_{t-1} \right), 
\end{equation}
is valid because the Lagrangian function solves a relaxed problem. This implies that the optimal $\pi_{t}^{ij}$ values attained in \eqref{prob:ldual} with $v_{t}^{ij} = \mc{L}_t^{i}(\pi_{t}^{ij}; \om_t^{j})$ also yield a valid cut. When the state variables are binary, \citealt{zou2019stochastic} prove that the cut obtained using $\pi$ values from \eqref{prob:ldual} is tight. This means that 
\begin{align} \label{tightLagrn}
\underline{Q}_t^i(x_{t-1}^{ik}, \psi_{t}^{i+1}, \om_t^j) = \max_{\pi_{t}^{ij}} \left( \mc{L}_t^{i}(\pi_{t}^{ij}; \om_t^{j}) + (\pi_{t}^{ij})^{\top} x_{t-1}^{ik}  \right).
\end{align}
when $\pi_t^{ij}$ is the optimal solution of \eqref{prob:ldual}.

\textbf{Strengthened Benders' cuts:} The strengthened Benders' cuts are of the form \eqref{cuts:lagn} where the $\pr{\pi_{t}^{ij}}_{j}$ are not obtained by solving the Lagrangian dual problem exactly but are instead optimal LP dual variables corresponding to the constraint $z_t = x_{t-1}^{ik}$. Strengthened Benders' cuts are known to be stronger than the standard Benders' cuts, but may not be tight.

We now show that Integer L-shaped cuts belong to the family of Lagrangian cuts. Specifically, consider the following feasible solution to the Lagrangian dual problem \eqref{prob:ldual}:
\begin{align} \label{intLpi}
   \pi^{ij}_{t,r} = \begin{cases}
        \underline{Q}_t^i(x_{t-1}^{ik}, \psi_{t}^{i+1}, \om_t^j) - L(\psi_{t}^{i+1}, \om_t^j), \quad r \in S(x_{t-1}^{ik}) \\
        L(\psi_{t}^{i+1}, \om_t^j) - \underline{Q}_t^i(x_{t-1}^{ik}, \psi_{t}^{i+1}, \om_t^j), \quad r \notin S(x_{t-1}^{ik}),
    \end{cases}
\end{align}
where the notation $v_r$ denotes the $r^{th}$ component of the vector $v$. The set $S(x_{t-1}^{ik})$ denotes the index of components of vector $x_{t-1}^{ik}$ which take value one. As before, the scalar $L(\psi_{t}^{i+1}, \om_t^j)$ denotes the lower bound on the expected future cost $ \underline{Q}_t^i(\cdot, \psi_{t}^{i+1}, \om_t^j)$. We now prove the following proposition:
\begin{proposition} \label{prop:intL}
    The vector $\pi$ defined in \eqref{intLpi} is optimal for \eqref{prob:ldual}.
\end{proposition}

\begin{proof}
    Let $(x_t^*, y_t^*, z_t^*)$ denote an optimal solution of the Lagrangian relaxation \eqref{prob:lag_function} with $\pi$ given as in \eqref{intLpi}. If $z_t^{*} = x_{t-1}^{ik}$, then the objective function in \eqref{prob:ldual} becomes:
    \begin{align} 
          &\mc{L}^{i}_{t}(\pi_t^{ij}, \om_t^j) + (\pi_t^{ij})^{\top} x_{t-1}^{ik} \nonumber \\
          =&f_t\left(x_t^*, y_t^*, \om_t^{j} \right)+\psi_t^{i+1}\left( x_t^* \right) - (\pi_t^{ij})^{\top} x_{t-1}^{ik} + (\pi_t^{ij})^{\top} x_{t-1}^{ik} \label{lagrn_relax_obj} \\
          =&f_t\left(x_t^*, y_t^*, \om_t^{j} \right)+\psi_t^{i+1}\left( x_t^* \right). \nonumber
    \end{align}
    This is exactly the optimal objective function value in \eqref{prob:sub+Copy_old}. The \textit{tightness} of the Lagrangian cut in \eqref{tightLagrn} then would imply that $\pi$ is optimal for \eqref{prob:ldual}. Therefore, it suffices to show that the proposed $\pi$ results in $z_t^{*} = x_{t-1}^{ik}$ in \eqref{prob:lag_function}. We show that for any $z_t$ not equal to $x_{t-1}^{ik}$ the objective function value in \eqref{prob:ldual} does not improve. In other words, the objective value either stays the same or worsens if we deviate from $z_t = x_t^{ik}$. For the feasible solution $z_t = x_{t-1}^{ik}$ and $\pi_t^{ij}$ as in \eqref{intLpi}, denote the objective function value obtained by optimizing \eqref{prob:lag_function} on remaining variables as $F_1 (x_t, y_t, w_t^j) + F_2(z_t = x_{t-1}^{ik})$, where $F_1$ captures the first two summands and $F_2$ captures the third summand in \eqref{lagrn_obj_old}. We refer to this objective value as the benchmark in the remaining proof. Now, consider the following cases, when $z_t$ is taken to be some other binary vector than $x_{t-1}^{ik}$.
    
    \textit{Case 1}: ${x_{t-1,p}^{ik}} = 1$ and $z_{t,p} = 0$ for some component $p$. The $F_2$ component in this case will increase by at least $\underline{Q}_t^i(x_{t-1}^{ik}, \psi_{t}^{i+1}, \om_t^j) - L(\psi_{t}^{i+1}, \om_t^j)$. However, the $F_1$ component may decrease. However, the decrease in $F_1$ component is at most $\underline{Q}_t^i(x_{t-1}^{ik}, \psi_{t}^{i+1}, \om_t^j) - L(\psi_{t}^{i+1}, \om_t^j)$. Such a choice of $z$, therefore, does not improve the objective against the benchmark.

    \textit{Case 2}: $x_{t-1,p}^{ik} =  0$ and $z_{t,p} = 1$ for some component $p$. Again, the $F_2$ component in this case will increase by at least $\underline{Q}_t^i(x_{t-1}^{ik}, \psi_{t}^{i+1}, \om_t^j) - L(\psi_{t}^{i+1}, \om_t^j)$. The $F_1$ component may decrease. However, the decrease in $F_1$ component is at most $\underline{Q}_t^i(x_{t-1}^{ik}, \psi_{t}^{i+1}, \om_t^j) - L(\psi_{t}^{i+1}, \om_t^j)$. Therefore, such a choice of $z$ also does not improve the objective against the benchmark.
    
    Either way we conclude that the proposed $\pi$ results in $z_t^{*} = x_{t-1}^{ik}$. Consequently, the vector $\pi$ defined in \eqref{intLpi} is optimal for \eqref{prob:ldual}.
\end{proof}

\begin{corollary}
    Integer L-shaped cuts \eqref{cut:intL} are one of the Lagrangian cuts \eqref{cuts:lagn}.
\end{corollary}
\begin{proof}
    We show that the coefficients defined in \eqref{intLpi} when substituted in \eqref{cuts:lagn}
    gives integer L-shaped cuts in \eqref{cut:intL}. 

Based on the tightness of the Lagrangian cut we have:
\begin{align}
    \mc{L}_t^{i}(\pi_{t}^{ij}; \om_t^{j}) &= \underline{Q}_t^i(x_{t-1}^{ik}, \psi_{t}^{i+1}, \om_t^j) - (\pi_{t}^{ij})^{\top} x_{t-1}^{ik} \nonumber \\
    &= \underline{Q}_t^i(x_{t-1}^{ik}, \psi_{t}^{i+1}, \om_t^j) - \left( \underline{Q}_t^i(x_{t-1}^{ik}, \psi_{t}^{i+1}, \om_t^j) - L(\psi_{t}^{i+1}, \om_t^j)\right) \left( \sum_{r: {x}_{t-1,r}^{ik} = 1} 1 \right) \label{lagrn_relax_obj2}
\end{align}
Next, the expression $(\pi_{t}^{ij})^{\top} x_{t-1}$ becomes
\begin{align}
    \left( \underline{Q}_t^i(x_{t-1}^{ik}, \psi_{t}^{i+1}, \om_t^j) - L(\psi_{t}^{i+1}, \om_t^j)\right) \left(\sum_{r: x_{t-1,r}^{ik} = 1} \left(x_{t-1,r} - 1\right) - \sum_{r: x_{t-1,r}^{ik} = 0} x_{t-1,r} \right). \label{coeff_lagrn}
\end{align}

Substituting \eqref{lagrn_relax_obj2} and \eqref{coeff_lagrn} in \eqref{cuts:lagn}, we obtain the integer L-shaped cut in \eqref{cut:intL}.
\end{proof}
We now adapt an example from \citealt{zou2019stochastic} to illustrate how Lagrangian cuts corresponding to two different optimal solutions of the Lagrangian dual problem result in cuts of varying strength.

\begin{example} 
    Consider the following two-stage stochastic integer program with one scenario,
    \begin{align*}
    \min_{x} \br{x_1 + x_2 + Q(x_1, x_2): x_1, x_2 \in \br{0,1}},
    \end{align*}
    where $Q(x_1, x_2) = \min \br{4y: y \geq 2.6 - 0.25 z_1 - 0.5 z_2, x_1 = z_1, x_2 = z_2, y \leq 4, y \in \mathbb{Z}_{+}, z_1, z_2 \in \br{0,1}}$. It is easy to see that $Q(x_1, x_2) \geq 8$ for any binary vector $(x_1, x_2)$. The incumbent solution is $(0,0)$ and $Q(0,0) = 12$. Now, both $(-4, -4)$ and $(0, -4)$ are optimal solutions to the Lagrangian dual problem. With $(-4, -4)$ as the Lagrange multipliers, we get the integer L-shaped cut $\theta \geq 12 - 4 x_1 - 4x_2$. With $(0, -4)$ as the Lagrange multipliers, we get a stronger cut $\theta \geq 12 - 4x_2$. Both cuts are tight but one is stronger than the other.
\end{example}

\subsection{Stopping criterion} \label{sec:cc}

The choice of the stopping criterion is important so that cut-generating strategies can be compared fairly. In our experiments, we use the stopping criterion of \citealt{homem2011sampling}, which views convergence from the perspective of testing the following hypothesis:
\begin{align*}
    H_0: UB \leq LB \text{ against } H_1: UB > LB.
\end{align*}
In hypothesis testing, a type I error occurs when $H_0$ is rejected even though it is true, and a type II error occurs when $H_0$ is retained even though it is false. In terms of the above hypothesis, type I error means that we continue with the iterations even though an optimal solution has been attained. Type II error means that we stop the algorithm even when there is a gap. \citealt{homem2011sampling} propose to control both the type-I and type-II errors with following pre-specified values: (i) bound $\alpha$ on probability of type-I error, (ii) bound $\gamma$ on probability of type-II error when $UB/LB > 1 + \delta$, and (iii) scalar $\delta \in [0,1]$ which can interpreted as the relative gap between deterministic lower and upper bounds. They show that this stopping criterion is more robust than some of the existing approaches (\citealt{pereira1991multi}). In our discussion of computational results, we refer to $\delta$ as pseudo-gap.

\section{Alternating cut criterion for solving MSIP problems}\label{sec:alt}

We now extend the alternating cut process of \citealt{angulo2016improving} for solving MSIP problems. The authors apply this technique to two-stage problems with integer L-shaped cuts. The approach can be naturally extended to multi-stage problems with any choice of tight cuts. Note that in backward step, the forward solution $\theta_{t-1}^{ik}$ always satisfies $\theta_{t-1}^{ik} \leq \sum_{j = 1}^{N_t} q_{j} \underline{Q}_{t}^{i}\left(x_{t-1}^{ik}, \psi_t^{i+1}, \omega_t^j \right)$. When $\theta_{t-1}^{ik} < \sum_{j = 1}^{N_t} q_{j} \underline{Q}_{t}^{i}\left(x_{t-1}^{ik}, \psi_t^{i+1}, \omega_t^j \right)$, then a tight cut of the form in \eqref{cut:form} is enforced to strengthen the subproblem, and cut off the incumbent solution. However, the evaluation of $\underline{Q}_{t}^{i}\left(x_{t-1}^{ik}, \psi_t^{i+1}, \omega_t^j \right)$ may require solving a computationally expensive integer program. Let ${Q}_{t}^{i, LP}\left(x_{t-1}^{ik}, \psi_t^{i+1}, \omega_t^j \right)$ denote the LP relaxation of the integer-program $P_{t}^{i}\left(x_{t-1}^{ik}, \psi_t^{i+1}, \omega_t^j \right)$. If the incumbent solution $\theta_{t-1}^{ik}$ also satisfies 
\begin{equation} \label{ineq:angulo}
    \theta_{t-1}^{ik} < \sum_{j = 1}^{N_t} q_{t}^{j} {Q}_{t}^{i,LP}\left(x_{t-1}^{ik}, \psi_t^{i+1}, \omega_t^j \right) \leq \sum_{j = 1}^{N_t} q_{j} \underline{Q}_{t}^{i}\left(x_{t-1}^{ik}, \psi_t^{i+1}, \omega_t^j \right),
\end{equation}
then the solution can be cut off by using the Benders' cut. Since the Benders' cut is much cheaper to compute, \citealt{angulo2016improving} propose only enforcing a Benders' cut if the first inequality in \eqref{ineq:angulo} holds. Otherwise, tight cuts such as integer L-shaped cuts or Lagrangian cuts are used. We describe the backward step of the SDDiP algorithm with alternating cut criterion in Algorithm \ref{algo:alt_cut}.

A key advantage of the alternating cut approach is that it can be used as a cut-generating method to further enhance some of the existing approaches (\citealt{gjelsvik1999algorithm, thome2013non, steeger2015strategic}, \citealt{rahmaniani2020benders, chen2022generating}). For instance, we can alternate between Benders' cuts and Lagrangian in the third phase of \citealt{rahmaniani2020benders}'s strategy. The solution times can be improved if Benders' cuts are used to cut off incumbent solutions rather than the computationally expensive Lagrangian cuts.

\begin{algorithm}[]
	\caption{Backward step of SDDiP with alternating cut criterion} 
    \label{algo:alt_cut}
	\begin{algorithmic}[1]
        \State \textbf{Input}: \multiline{Sampled scenarios: $\pr{\om_1^{k}, \ldots, \om_{T}^{k}}$, the solution from the forward step $\pr{x_t^{ik}}_{t, k}$  and approximation $\pr{\psi_t^{i}}_{t}$ of the expected cost-to-go function.}
        \For {$t = T, \ldots, 2$}
            \For{$k = 1, \ldots, M$}
                \For {$j = 1, \ldots, N_t$}
                        \State \multiline{%
                                        Solve a linear programming relaxation $R_t^{ij}$ of the updated problem $P_{t}^{i}(x_{t-1}^{ik}, \psi_t^{i+1}, \om_t^{j})$ and collect Benders'  coefficients $(v_t^{ij}, \pi_{t}^{ij})$.}
                \EndFor
                \If {$\theta_{t-1}^{ik} < \sum_{j = 1}^{N_t} q_{j} {Q}_{t}^{i,LP}\left(x_{t-1}^{ik}, \psi_t^{i+1}, \omega_t^j \right)$}
                    \State \multiline{Add Benders'  cut using the coefficients $\br{(v_t^{ij}, \pi_t^{ij})}_{j = 1, \ldots, N_t}$ to $\psi_{t-1}^i$ to get $\psi_{t-1}^{i+1}$.}
                \Else
                    \For {$j = 1, \ldots, N_t$}
                        \State \multiline{%
                                        Solve a `stronger' relaxation $R_t^{ij}$ of the updated problem $P_{t}^{i}(x_{t-1}^{ik}, \psi_t^{i+1}, \om_t^{j})$ and collect coefficients $(v_t^{ij}, \pi_{t}^{ij})$.}
                    \EndFor
                \State \multiline{Add optimality cut using the coefficients $\br{(v_t^{ij}, \pi_t^{ij})}_{j = 1, \ldots, N_t}$ to $\psi_{t-1}^i$ to get $\psi_{t-1}^{i+1}$.}
                \EndIf
            \EndFor
        \EndFor
	\end{algorithmic} 
\end{algorithm}

\section{Computational experiments} \label{sec:comp}

We present a computational study to investigate how the alternating cuts perform for multi-stage stochastic integer programming problems. The goal of the study is three-fold.  First, we aim to evaluate how decomposition methods for solving MSIPs compare against solving a large-scale extensive form. Second, we study the impact of the alternating cut process on the SDDiP algorithm. Specifically, we implement the SDDiP Algorithm \ref{algo:sddip} with the backward step given in Algorithm \ref{algo:alt_cut}, and the convergence criterion discussed in Section \ref{sec:cc}. We test the proposed algorithm on two classes of real-world problems. We consider the multi-stage version of the stochastic multi-knapsack problem (SMKP) in \citealt{angulo2016improving} and the generation expansion planning problem in \citealt{zou2019stochastic}. The structure of these problems is discussed in Appendix \ref{sec:app-problems}. We use the same data-generating procedure as in the original papers. Finally, we evaluate how different choices of parameters in the SDDiP algorithm, such as the number of paths sampled, pseudo-gap $\delta$ in the termination criterion, and tolerance probabilities $\alpha, \gamma$, impact the performance of the SDDiP algorithm. The alternating cut approach and Nested Benders' algorithm are implemented in the SDDP package (\citealt{dowson2021sddp}) in Julia. All experiments are conducted on an Intel Xeon Gold 5218R CPU using 2 cores and 6 GB memory per core. The code for the algorithm is available at \url{https://github.com/akulbansal5/sddp_ac/tree/master} and the datasets and problem formulations are available in the repository \url{https://github.com/akulbansal5/sddp_data/tree/master}.

\subsection{Stochastic multi-knapsack problem}

We begin by comparing the Nested Benders' algorithm against solving the extensive form of the MSIP problem. Specifically, we compare three approaches: classical Nested Benders' algorithm, Nested Benders' with the alternating cut procedure, and solving the extensive form as a large integer program. Because the gaps in the SDDiP algorithm are only estimates, testing the alternating cut procedure with the Nested Benders' algorithm allows us to compare the exact optimality gaps at termination with those in the extensive form.  We use SMKP for our test instances in this experiment. SMKP has binary variables which means that Nested Benders' algorithm converges to an optimal solution of the MSIP problem with tight cuts, such as Lagrangian cuts. We use a time limit of one hour and a gap threshold of 1\% for termination. We compare the performance with respect to the solution times and the relative gap percentage at termination. All reported metrics are obtained by averaging results across five instances from each instance class defined by the number of stages ($T$), number of knapsack constraints (rows), number of knapsack items (cols), and number of scenarios per node of the tree (scens).

In Figure \ref{fig:smkp_nbvsde}, we compare the Nested Benders' algorithm, which utilizes Integer L-shaped cuts (indicated by \textit{cut=I}), with the alternating cut procedure (indicated by \textit{bpass=A}) and the extensive form solved as a large integer program. The sub-figures on the left depict the solution times and those on the right depict the relative gap percentage. Each plot shows results for instances with a specific number of stages, $T$, indicated in the plot title. Further, each curve in the plot corresponds to instances with a specific number of scenarios per stage. This is indicated in the plot legend. $\text{Ben\_}v$ indicates the results for the decomposition algorithm when applied to instances with $v$ scenarios per stage. Similarly, $\text{DE\_}v$ indicates the results for the deterministic equivalent for instances with $v$ scenarios per stage. The tuples on the $x$-axis represent the number of knapsack constraints and number of knapsack items in each subproblem of the instance. We observe from Figure \ref{fig:smkp_nbvsde} that as the instance size increases in the number of stages or scenarios per stage, the decomposition method performs better than the deterministic equivalent on solution times and relative gap percentage. We also notice that solution times and gaps increase for instances with a larger number of knapsack constraints and a smaller number of knapsack items. 

In Figure \ref{fig:smkp_defvsalt}, we compare the classical Nested Benders' algorithm with the version that also alternates with Benders cuts. In the plot legends, $\text{Def}\_v$ and $\text{Alt}\_v$ denote the default approach and the alternating approach for instances with $v$ scenarios per stage. We observe that the default approach times out in all cases with termination gaps at more than 40\%. On the other hand, with the alternating approach, the algorithm terminates before the time limit in most cases and the gaps are closer to 1\%. These results show that the alternating cut generation method outperforms the default cut generation approach and the deterministic equivalent.

\begin{figure}[H]
\centering
\begin{minipage}{.5\textwidth}
  \centering
  \caption*{(a): T = 3}

  \vspace{-0.6cm}
  \includegraphics[scale = 0.55]{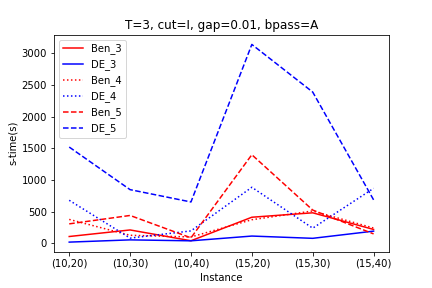}

\end{minipage}%
\begin{minipage}{.5\textwidth}
  \centering
  \caption*{(d) T = 3}
  \vspace{-0.6cm}
  \includegraphics[scale = 0.55]{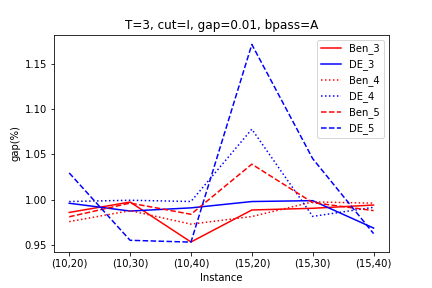}
\end{minipage}

\begin{minipage}{.5\textwidth}
  \centering
  \caption*{(b): T = 4}
  \vspace{-0.6cm}
  \includegraphics[scale = 0.55]{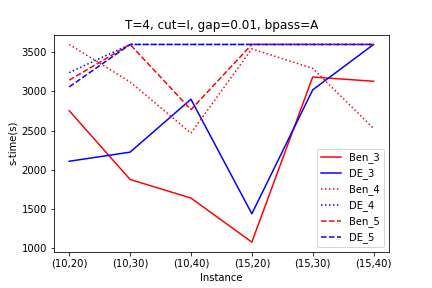}
\end{minipage}%
\begin{minipage}{.5\textwidth}
  \centering
  \caption*{(e): T = 4}
  \vspace{-0.6cm}
  \includegraphics[scale = 0.55]{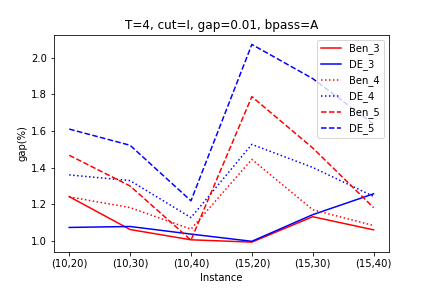}
\end{minipage}

\begin{minipage}{.5\textwidth}
  \centering
  \caption*{(c): T = 5}
  \vspace{-0.6cm}
  \includegraphics[scale = 0.55]{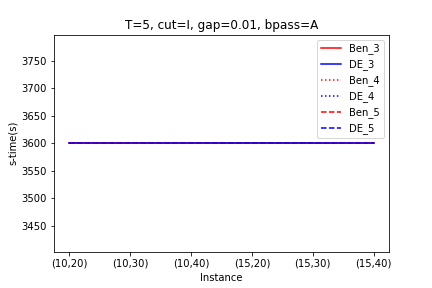}
\end{minipage}%
\begin{minipage}{.5\textwidth}
  \centering
  \caption*{(f): T = 5}
  \vspace{-0.6cm}
  \includegraphics[scale = 0.55]{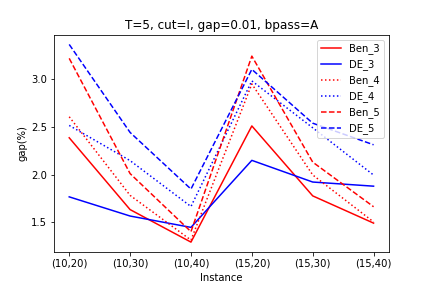}
\end{minipage}
\caption{Solution times (in seconds) and relative gap percent (\%) for the deterministic equivalent and the Nested Benders algorithm alternating between Benders' cuts and Integer L-shaped cuts. {In Figure (c) all instances hit the time limit.}}
\label{fig:smkp_nbvsde}

\end{figure}

\begin{figure}[H]

\centering
\begin{minipage}{.5\textwidth}
  \centering
  \caption*{(a): T = 3}
  \vspace{-0.6cm}
  \includegraphics[scale = 0.55]{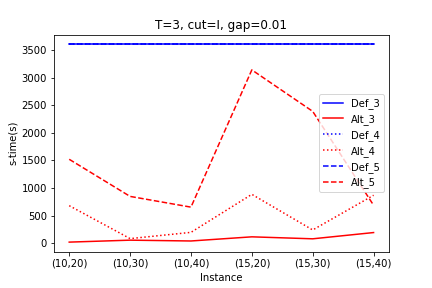}
\end{minipage}%
\begin{minipage}{.5\textwidth}
  \centering
  \caption*{(d): T = 3}
  \vspace{-0.6cm}
  \includegraphics[scale = 0.55]{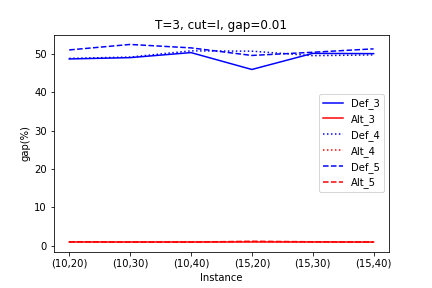}
\end{minipage}

\begin{minipage}{.5\textwidth}
  \centering
  \caption*{(b): T = 4}
  \vspace{-0.6cm}
  \includegraphics[scale = 0.55]{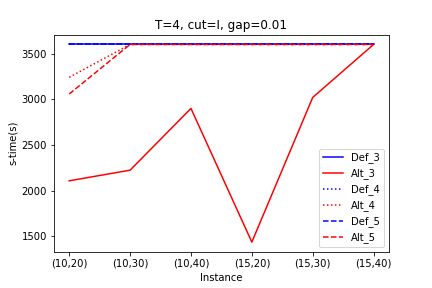}
\end{minipage}%
\begin{minipage}{.5\textwidth}
  \centering
  \caption*{(e): T = 4}
  \vspace{-0.6cm}
  \includegraphics[scale = 0.55]{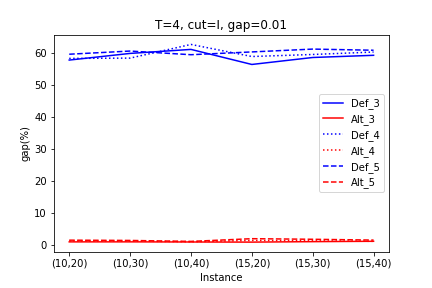}
\end{minipage}

\begin{minipage}{.5\textwidth}
  \centering
  \caption*{(c): T = 5}
  \vspace{-0.6cm}
  \includegraphics[scale = 0.55]{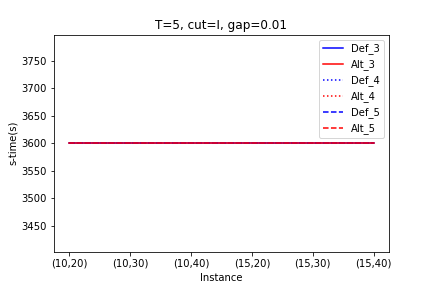}
\end{minipage}%
\begin{minipage}{.5\textwidth}
  \centering
  \caption*{(f): T = 5}
  \vspace{-0.6cm}
  \includegraphics[scale = 0.55]{DefaultAltOnT_gap_4_I_0.01.png}
\end{minipage}

\caption{Solution times (in seconds) and relative gap percent for the Nested Benders algorithm with the default and alternating cut approach. The alternating method utilizes Benders and Integer L-shaped cuts. {In Figure (c), all instances hit the time limit.}}
\label{fig:smkp_defvsalt}

\end{figure}

Next, we evaluate the impact of our proposed Algorithm \ref{algo:alt_cut} on the SDDiP method. We first report the results for SMKP instances. For termination of the SDDiP algorithm, we use the convergence criterion discussed in Section \ref{sec:cc} with pseudo-gap $\delta = 0.01$, and tolerance probabilities $\gamma = 0.10$ and $\alpha = 0.10$. In our study, we also experimented with different values of these parameters. Because the trend in the metrics we compare is similar for all parameter choices, we only report results for a particular choice of parameters. {To account for uncertainty in the data-generating process, we consider two instances from each instance class. Further, to account for uncertainty in the SDDiP algorithm, we consider two seeds in the random sampling of the scenario paths for each instance.} After termination, we sample 5\% of scenario paths independently, to compute a confidence interval on the upper bound. We then report the average value of four independent runs for each instance class in Tables \ref{tab:smkp_time}--\ref{tab:smkp_prop}.  The column \textbf{M} denotes the number of scenario paths sampled. {Some of the column names have the form \textbf{xy-z} where \textbf{x} denotes one of the tight cuts, that is, a Lagrangian cut obtained using a subgradient algorithm (L) or Integer L-shaped cut (I). The subgradient algorithm we use to solve the Lagrangian dual problem is a modified version of Broyden–Fletcher–Goldfarb–Shanno (BFGS) algorithm (\citealt{nocedal1999numerical}) implemented in the SDDP package (\citealt{dowson2021sddp}). The \textbf{y} in \textbf{xy-z} denotes the cut-generating procedure: either the default (D) method or the alternating cut (A) approach.} The extension \textbf{z} denotes the performance metric (time, iterations, proportion of tight cuts used, and gap estimate). We use notation \textbf{t} to denote time, \textbf{iter} to denote the iterations, \textbf{Aprop} to denote the proportion of tight cuts, and \textbf{gap(\%)} to denote the optimality gap estimate. The gap estimate is obtained by computing the relative gap between the lower bound on termination and the right end of the confidence interval on the upper bound.

We now discuss results for the SMKP instances in Tables \ref{tab:smkp_time}--\ref{tab:smkp_prop}.  From Table \ref{tab:smkp_time}, we see that the default approach always hits the time limit. Further, alternating with Lagrangian cuts tends to take more time compared to using integer L-shaped cuts. This is expected because deriving Lagrangian cuts using a subgradient algorithm is computationally expensive. This is further supported by the observation from Table \ref{tab:smkp_iter}, that alternating with integer L-shaped cuts allows more iterations in the same or less time than with Lagrangian cuts. Because the choice of the method for solving the Lagrangian dual problem plays a crucial role in determining the strength of the cut and computational effort required to derive it, we also experiment with \citealt{chen2022generating}'s method for deriving the Lagrangian cut. Our observations 
are discussed in Appendix \ref{sec:chen}.

Next, from Table \ref{tab:smkp_gaps}, we notice that the alternating cut approach performs better than the default approach across all instance classes and $M \in \{2, 5\}$ in terms of gaps. Also, the gaps attained from using Lagrangian cuts are comparable with those from Integer L-shaped cuts. The advantage of the alternating approach can also be seen in Table \ref{tab:smkp_prop}. On average, the tight cuts are fewer than 20\% of the total number of cuts. This allows us to accelerate the convergence of the algorithm by using computationally less-intensive Benders' cut the remaining 80\% of the time.

\subsection{Generation expansion planning}

Next, we discuss our computational experiments with the same GEP instances as in \citealt{zou2019stochastic}. The instance is based on real-life data taken from \citealt{jin2011modeling}. The precise formulation is given in Appendix \ref{sec:app-problems}. The problem has 10 decision stages with three equally likely possible scenarios in each stage. The total number of scenarios is $3^9 = 19,683$. We consider three scenario trees from the underlying distribution and solve each of them three times with different random seeds. Owing to the large scale of the instance, the build times for the extensive form are prohibitive and gaps with a one-hour time limit are more than 10\%. For different choices of parameters, we report the average value of the total nine independent runs for each parameter setting. The results are reported in Tables \ref{tab:gep-times}--\ref{tab:gep-gaps}. The columns $\delta$, $M$, and $\alpha, \gamma$ denote the choice of pseudo-gap, number of sampled paths, and tolerance probabilities in the stopping criterion, respectively. The names in other columns are the same as in the case of SMKP instances.

{In Table \ref{tab:gep-times}, we record the solution times (in seconds) for the SDDiP algorithm with different choices of cuts and cut-generating methods. We observe that the alternating cut method outperforms the default method in most cases. We see a noticeable improvement in solution times when alternating with Integer L-shaped cuts. It is well known that integer L-shaped cuts tend to be weak and the decomposition algorithm with the default approach struggles when we only utilize cuts from this family. Next, we observe that with Lagrangian cuts obtained using the subgradient algorithm, the performance of the default and alternating cut approach is comparable, although, the alternating method still outperforms the default method in most cases. From Table \ref{tab:gep-iter}, we observe that the alternating cut generating method takes more iterations to terminate than the default method in the case of Lagrangian cuts. Despite this, we see an overall improvement in the solution times in Table \ref{tab:gep-times}.
The gap estimates in Table \ref{tab:gep-gaps} are similar for the Lagrangian cuts with and without the alternating-cut generating approach. However, there is a significant improvement in the gap estimates when we alternate with Integer L-shaped cuts. Overall, the default cut-generating approach with Lagrangian cuts outperforms the other combinations, but the gaps with the alternating method are also comparable. Unsurprisingly, this indicates that the Lagrangian cuts obtained using the subgradient method are stronger than the Integer L-shaped cuts which belong to the same family. This is further confirmed in Table \ref{tab:gep-prop} by the proportion of tight cuts we add in the alternating cut method. The proportion of Lagrangian cuts used is less than the proportion of integer L-shaped cuts in all cases.}

Turning to the matter of parameter choice, we now study how different parameters impact the algorithm's performance. {From Tables \ref{tab:smkp_time}, \ref{tab:smkp_gaps}, \ref{tab:gep-times} and \ref{tab:gep-gaps}, we observe both in SMKP and GEP instances, that as the value of $M$ increases, the gaps improve marginally but at the cost of significant overhead in solution times. Next, we observe for GEP instances (Table \ref{tab:gep-times} and \ref{tab:gep-gaps}) that there is no significant difference in solution times or estimated gaps when tolerance probabilities are chosen to be 0.10 or 0.15.} The choice of the pseudo-gap, however, does have a noticeable impact with $\delta = 0.01$ yielding smaller gaps and higher solution times on average than $\delta = 0.05$. Overall, the results indicate that a properly tuned decomposition algorithm with the alternating cut procedure can provide comparable performance in less time than using the default cut-generating procedure.

\begin{table}[]
\centering

\caption{Solution times (in seconds) for SMKP instances and $M \in \{2,5\}$}
\label{tab:smkp_time}
\resizebox{0.65\textwidth}{!}{%
\begin{tabular}{c|cccc|cc|cc}
\textbf{\textbf{M}} & \textbf{\textbf{T}} & \textbf{\textbf{rows}} & \textbf{\textbf{cols}} & \textbf{\textbf{scens}} & \textbf{\textbf{LD-t}} & \textbf{\textbf{ID-t}} & \textbf{\textbf{LA-t}} & \textbf{\textbf{IA-t}} \\ \hline
2 & 3 & 10 & 30 & 3 & $>$3600 & $>$3600 & 97.73 & 37.37 \\
5 & 3 & 10 & 30 & 3 & $>$3600 & $>$3600 & 3184.44 & 2742.61 \\
2 & 3 & 10 & 30 & 5 & $>$3600 & $>$3600 & 625.24 & 42.13 \\
5 & 3 & 10 & 30 & 5 & $>$3600 & $>$3600 & 2760.35 & 1340.29 \\
2 & 4 & 10 & 30 & 3 & $>$3600 & $>$3600 & 336.98 & 56.41 \\
5 & 4 & 10 & 30 & 3 & $>$3600 & $>$3600 & $>$3600 & 2188.28 \\
2 & 4 & 10 & 30 & 10 & $>$3600 & $>$3600 & 1267.86 & 129.6 \\
5 & 4 & 10 & 30 & 10 & $>$3600 & $>$3600 & $>$3600 & 3197.54 \\
2 & 5 & 15 & 40 & 3 & $>$3600 & $>$3600 & 84.17 & 41.61 \\
5 & 5 & 15 & 40 & 3 & $>$3600 & $>$3600 & $>$3600 & $>$3600 \\
2 & 5 & 15 & 40 & 10 &$>$3600 & $>$3600 & 395.9 & 221.77 \\
5 & 5 & 15 & 40 & 10 & $>$3600 & $>$3600 & $>$3600 & $>$3600
\end{tabular}%
}
\end{table}

\begin{table}[]
\centering
\caption{Iterations to terminate for SMKP instances and $M \in \br{2,5}$}
\label{tab:smkp_iter}
\resizebox{0.65\textwidth}{!}{%
\begin{tabular}{c|cccc|cc|cc}
\textbf{\textbf{M}} & \textbf{\textbf{T}} & \textbf{\textbf{rows}} & \textbf{\textbf{cols}} & \textbf{\textbf{scens}} & \textbf{\textbf{LD-iter}} & \textbf{\textbf{ID-iter}} & \textbf{\textbf{LA-iter}} & \textbf{\textbf{IA-iter}} \\ \hline
2 & 3 & 10 & 30 & 3 & 31 & 323 & 13 & 13 \\
5 & 3 & 10 & 30 & 3 & 29 & 297 & 71 & 240 \\
2 & 3 & 10 & 30 & 5 & 25 & 287 & 24 & 24 \\
5 & 3 & 10 & 30 & 5 & 19 & 241 & 34 & 163 \\
2 & 4 & 10 & 30 & 3 & 20 & 268 & 23 & 23 \\
5 & 4 & 10 & 30 & 3 & 12 & 203 & 52 & 216 \\
2 & 4 & 10 & 30 & 10 & 10 & 224 & 47 & 53 \\
5 & 4 & 10 & 30 & 10 & 5 & 126 & 39 & 206 \\
2 & 5 & 15 & 40 & 3 & 14 & 187 & 10 & 10 \\
5 & 5 & 15 & 40 & 3 & 8 & 137 & 62 & 185 \\
2 & 5 & 15 & 40 & 10 & 5 & 128 & 37 & 37 \\
5 & 5 & 15 & 40 & 10 & 2 & 85 & 64 & 152
\end{tabular}%
}
\end{table}

\begin{table}[]
\centering
\caption{Gap estimates at termination for SMKP instances and two choices of M}
\label{tab:smkp_gaps}
\resizebox{0.80\textwidth}{!}{%
\begin{tabular}{c|cccc|cc|cc}
\textbf{\textbf{M}} & \textbf{\textbf{T}} & \textbf{\textbf{rows}} & \textbf{\textbf{cols}} & \textbf{\textbf{scens}} & \textbf{\textbf{LD-gap(\%)}} & \textbf{\textbf{ID-gap(\%)}} & \textbf{\textbf{LA-gap(\%)}} & \textbf{\textbf{IA-gap(\%)}} \\ \hline
2 & 3 & 10 & 30 & 3 & 52.795 & 52.292 & 3.357 & 3.358 \\
5 & 3 & 10 & 30 & 3 & 52.102 & 52.502 & 3.396 & 3.495 \\
2 & 3 & 10 & 30 & 5 & 54.755 & 54.022 & 4.004 & 4.069 \\
5 & 3 & 10 & 30 & 5 & 54.181 & 55.828 & 3.836 & 3.744 \\
2 & 4 & 10 & 30 & 3 & 61.806 & 62.19 & 4.877 & 4.877 \\
5 & 4 & 10 & 30 & 3 & 60.94 & 61.491 & 4.132 & 4.256 \\
2 & 4 & 10 & 30 & 10 & 60.549 & 65.523 & 3.204 & 3.068 \\
5 & 4 & 10 & 30 & 10 & 60.903 & 71.722 & 3.104 & 3.115 \\
2 & 5 & 15 & 40 & 3 & 68.076 & 68.011 & 3.411 & 3.411 \\
5 & 5 & 15 & 40 & 3 & 68.229 & 68.129 & 2.96 & 3.076 \\
2 & 5 & 15 & 40 & 10 & 68.338 & 68.131 & 2.397 & 2.384 \\
5 & 5 & 15 & 40 & 10 & 67.958 & 68.256 & 2.178 & 2.059
\end{tabular}%
}
\end{table}

\begin{table}[]
\centering
\caption{Proportion of tight cuts at termination for SMKP instances and $M \in \br{2,5}$}
\label{tab:smkp_prop}
    \resizebox{0.60\textwidth}{!}{%
\begin{tabular}{c|cccc|cc}
\textbf{\textbf{M}} & \textbf{\textbf{T}} & \textbf{\textbf{rows}} & \textbf{\textbf{cols}} & \textbf{\textbf{scens}} & \textbf{\textbf{LA-eprop}} & \textbf{\textbf{IA-eprop}}\\  \hline
2 & 3 & 10 & 30 & 3 & 0.123 & 0.123 \\
5 & 3 & 10 & 30 & 3 & 0.331 & 0.435 \\
2 & 3 & 10 & 30 & 5 & 0.157 & 0.157 \\
5 & 3 & 10 & 30 & 5 & 0.197 & 0.416 \\
2 & 4 & 10 & 30 & 3 & 0.135 & 0.135 \\
5 & 4 & 10 & 30 & 3 & 0.247 & 0.371 \\
2 & 4 & 10 & 30 & 10 & 0.066 & 0.081 \\
5 & 4 & 10 & 30 & 10 & 0.110 & 0.219 \\
2 & 5 & 15 & 40 & 3 & 0.015 & 0.015 \\
5 & 5 & 15 & 40 & 3 & 0.108 & 0.194 \\
2 & 5 & 15 & 40 & 10 & 0.004 & 0.004 \\
5 & 5 & 15 & 40 & 10 & 0.03 & 0.073
\end{tabular}%
}
\end{table}

\begin{table}[]
\centering
\caption{Solution times under different parameter choices for GEP instances}
\label{tab:gep-times}
\resizebox{0.60\textwidth}{!}{%
\begin{tabular}{ccc|cc|cc}
\textbf{\textbf{$\delta$}} & \textbf{\textbf{M}} & \textbf{\textbf{$\alpha$,  $\gamma$}} & \textbf{\textbf{LD-t}} & \textbf{\textbf{ID-t}} & \textbf{\textbf{LA-t}} & \textbf{\textbf{IA-t}} \\ \hline
0.01 & 2 & 0.15 & 64.28 & $>$3600 & 71.93 & 50.71 \\
0.01 & 5 & 0.15 & 3149.33 & $>$3600 & 2832.64 & $>$3600 \\
0.01 & 2 & 0.10 & 70.13 & $>$3600 & 66.79 & 57.25 \\
0.01 & 5 & 0.10 & $>$3600 & $>$3600 & 3436.96 & $>$3600 \\
0.05 & 2 & 0.15 & 32.40 & $>$3600 & 24.10 & 19.28 \\
0.05 & 5 & 0.15 & 63.15 & $>$3600 & 53.12 & 55.59 \\
0.05 & 2 & 0.10 & 32.06 & $>$3600 & 25.70 & 19.86 \\
0.05 & 5 & 0.10 & 73.09 & $>$3600 & 81.91 & 116.22
\end{tabular}%
}
\end{table}

\begin{table}[]
\centering
\caption{Iterations under different parameter choices for GEP instances}
\label{tab:gep-iter}
\resizebox{0.55\linewidth}{!}{%
\begin{tabular}{ccc|cc|cc}
\textbf{\textbf{$\delta$}} & \textbf{\textbf{M}} & \textbf{\textbf{$\alpha$,  $\gamma$}} & \textbf{\textbf{LD-iter}} & \textbf{\textbf{ID-iter}} & \textbf{\textbf{LA-iter}} & \textbf{\textbf{IA-iter}} \\ \hline
0.01 & 2 & 0.15 & 35 & 896 & 46 & 42 \\
0.01 & 5 & 0.15 & 908 & 409 & 756 & 298 \\
0.01 & 2 & 0.10 & 40 & 876 & 47 & 51 \\
0.01 & 5 & 0.10 & 1046 & 408 & 976 & 294 \\
0.05 & 2 & 0.15 & 9 & 870 & 10 & 10 \\
0.05 & 5 & 0.15 & 16 & 415 & 18 & 20 \\
0.05 & 2 & 0.10 & 10 & 870 & 9 & 12 \\
0.05 & 5 & 0.10 & 18 & 401 & 26 & 32
\end{tabular}%
}
\end{table}

\begin{table}[]
\centering
\caption{Gaps under different parameter choices for GEP instances}
\label{tab:gep-gaps}
\resizebox{0.75\textwidth}{!}{%
\begin{tabular}{ccc|cc|cc}
\textbf{\textbf{\textbf{$\delta$}}} & \textbf{\textbf{\textbf{M}}} & \textbf{\textbf{\textbf{$\alpha$,  $\gamma$}}} & \textbf{\textbf{\textbf{LD-gap(\%)}}} & \textbf{\textbf{ID-gap(\%)}} & \textbf{\textbf{\textbf{LA-gap(\%)}}} & \textbf{\textbf{IA-gap(\%)}} \\ \hline
0.01 & 2 & 0.15 & 0.865 & 84.857 & 1.091 & 2.595 \\
0.01 & 5 & 0.15 & 0.881 & 87.666 & 0.891 & 0.414 \\
0.01 & 2 & 0.10 & 0.868 & 85.386 & 0.992 & 2.329 \\
0.01 & 5 & 0.10 & 0.881 & 87.457 & 0.893 & 0.414 \\
0.05 & 2 & 0.15 & 1.170 & 84.827 & 1.737 & 2.573 \\
0.05 & 5 & 0.15 & 1.078 & 87.598 & 1.391 & 2.277 \\
0.05 & 2 & 0.10 & 1.081 & 84.659 & 1.770 & 2.599 \\
0.05 & 5 & 0.10 & 1.078 & 87.661 & 1.270 & 2.170
\end{tabular}%
}
\end{table}

\begin{table}[]
\centering
\caption{Proportion of tight cuts used under different parameter choices for GEP instances}
\label{tab:gep-prop}
\resizebox{0.5\textwidth}{!}{%
\begin{tabular}{ccc|cc}
\textbf{\textbf{$\delta$}} & \textbf{\textbf{M}} & \textbf{\textbf{$\alpha$,  $\gamma$}} & \textbf{\textbf{LA-eprop}} & \textbf{\textbf{IA-eprop}} \\ \hline
0.01 & 2 & 0.15 & 0.317 & 0.437 \\
0.01 & 5 & 0.15 & 0.356 & 0.661 \\
0.01 & 2 & 0.10 & 0.314 & 0.451 \\
0.01 & 5 & 0.10 & 0.364 & 0.661 \\
0.05 & 2 & 0.15 & 0.101 & 0.146 \\
0.05 & 5 & 0.15 & 0.209 & 0.382 \\
0.05 & 2 & 0.10 & 0.089 & 0.159 \\
0.05 & 5 & 0.10 & 0.241 & 0.483
\end{tabular}%
}
\end{table}

\section{Conclusions}\label{sec:concl}

We perform a computational study of cutting plane methods for solving multi-stage stochastic integer programming problems. We extend the alternating cut strategy from two-stage problems to the multi-stage setup. In our application of the proposed strategy, we consider alternating between any choice of valid and tight cut. We show that the Lagrangian dual problem may have multiple optimal solutions which lead to cuts of varying strengths. Further, we conduct a computational study with two classes of problems and report encouraging results that show the merit of the alternating procedure in multi-stage settings.

\section*{Acknowledgments} This work is supported, in part, by NSF \#2007814 and ONR \#N00014-22-1-2602.

\appendix

\section{Problem Classes}\label{sec:app-problems}

In this section, we provide a more detailed discussion of the problem classes we considered. Two-stage SMKP instances are generated in \citealt{angulo2016improving}. We extend the problem to multi-stage as follows. The stage $t$ subproblem with $j^{th}$ uncertainty realization as $w_t^j$ is given as follows:
\begin{align}
    Q_t(x_{t-1}, w_t^j) := & \min_{x_t, y_t} \hspace{0.2cm} (q(w_t^j))^{\top} x_t + \rho\mathbbm{1}^{\top} y_t + \sum_{k=1}^{N_{t+1}}q_k Q_{t+1}(x_t, w_{t+1}^k) \nonumber \\ 
    \text { s.t. } & W_{t-1} x_{t-1} + A_t x_t + I y_t \geq h_t, \label{cons:smkp_sub_cons}\\ 
    & x^{t} \in\{0,1\}^n, y_t \geq 0. \nonumber
\end{align}
Note that the uncertainty only affects the objective function through the coefficients $q(w_t^j)$. {This is the minimization version of the stochastic knapsack problem with constraint \eqref{cons:smkp_sub_cons} enforcing that items picked in stage $t$ must have at least value $h_t$. The incumbent solution $x_0$ in the first stage is also treated as a binary variable rather than a vector whose value is given as input.} The notation $I$ denotes the identity matrix and $\mathbbm{1}$ denotes a vector of all ones. The local variable vector $y_t$ captures the violation in one or more knapsack constraints. Each violation is penalized with cost $\rho$ in the objective. The elements for matrices $A_t$, $W_{t-1}$ and vector $q(w_t^j)$ are sampled uniformly from $\{1,..., 100\}$. The vector $h_t$ equals $\frac{3}{4} W_{t-1} \mathbbm{1} + \frac{3}{4} A_t \mathbbm{1}$. The scalar cost $\rho$ is taken to be 200.

Next, we discuss the generation expansion planning (GEP) problem (\citealt{zou2019stochastic, jin2011modeling}). There are $n$ types of generators that are available. Let $g_t$ and $x_t$ denote the decision vectors indicating the number of generators of each type built at stage $t$, and the cumulative number of generators of each type built until stage $t$, respectively. Each time period $t$ is further divided into sub-periods. Following \citealt{jin2011modeling}, we consider three sub-periods indicating low, medium, and high demand. Let $y_{ts}$ be the local variable vector denoting the electricity produced by each generator in time period $t$ and sub-period $s$. The demand in stage $t$ and sub-period $s$ is denoted by $d_{ts}$ while the unmet demand is given by $u_{ts}$. Now, the optimization problem at stage $t$ is given by:
\begin{subequations}
\begin{align}
    Q_t(x_{t-1}, w_t^j) := & \min_{g_t, y_{ts}} \hspace{0.2cm} a_t^{\top} g_t + \sum_{s} (q_s(w_t^j))^{\top} y_{ts}+ \rho\sum_{s} u_{ts} + \sum_{k=1}^{N_{t+1}}q_k Q_{t+1}(x_t, w_{t+1}^k)  \\ 
    &x_t \leq G \label{eq:gep_bound}\\ 
    &x_t = x_{t-1} + g_t \label{eq:gep_connect}\\
    &y_{ts} \leq A_t x_t, &s \in S \label{eq:gep_elec}\\
    &\mathbf{1}^{\top} y_{ts} + u_{ts} = d_{ts}(w_t^j), &s \in S \label{eq:gep_dem}\\
    &x_t, g_t \in \mathbb{Z}_{+}^{n}, y_t \in \mathbb{R}^{n}_{+}. \nonumber
\end{align}
\end{subequations}

The uncertainty appears in the objective function through the coefficient vector $q_s(w_t^j)$. It also appears in constraints because demand $d_{ts}$ is random. The constraint \eqref{eq:gep_bound} enforces that only a pre-determined number of generators $G$ can be used across all time periods. In constraint \eqref{eq:gep_connect}, the decision variables in the current stage connect with those in previous stages. The constraint \eqref{eq:gep_elec} limits the generation capacity based on the available number of generators. The matrix $A_t$ contains the maximum rating and maximum capacity information of generators. Finally, the constraint \eqref{eq:gep_dem} enforces that the demand is met. The objective function evaluates the cost of adding generators and generating electricity while penalizing unmet demand. The first summand captures the cost of building the generators through coefficients $a_t$, the second summand is the cost of generating the electricity, and the third summand is the penalty cost of not meeting the demand. The final summand is the expected future cost arising in future stages.

\section{Restricted separation of Lagrangian cuts} \label{sec:chen}

In their proposed approach, \citealt{chen2022generating} add cuts in a multi-cut setting where a separate cut is added to approximate future cost associated with each scenario instead of aggregating all cuts as in \eqref{cut:form}. Therefore, for two-stage problems, the master problem contains variables of the form $\theta_{\om^{j}}$ for $j \in \br{1, \ldots, N_2}$, where $N_2$ is the number of uncertainty realizations in the second stage. Now, for a given scenario $\om \in \br{\om^{1}, \ldots, \om^{N_2}}$ and given coefficients $(\lambda_{\om}, \lambda_{0, \om})$ belonging to a set $\Pi_{\omega}$, the cuts in \citealt{chen2022generating} are of the form 
\begin{align} \label{cut:chen_lagrn}
\lambda_{\om}^{\top}x + \lambda_{0,\om} \theta_{\om} \geq \mc{L}_{\om}(\lambda_{\om}, \lambda_{0, \om}), 
\end{align}
where $x$ is the state vector from the previous stage, and $\mc{L}_{\om}(\lambda_{\om}, \lambda_{0, \om})$ is the optimal objective value of the Lagrangian relaxation of the subproblem associated with scenario $\om$. In the Lagrangian relaxation, the non-anticipativity constraints are relaxed. For a suitable $\Pi_{\om}$ stated in \citealt[Proposition 1]{chen2022generating}, the cuts of the form \eqref{cut:chen_lagrn} are the same as Lagrangian cuts \eqref{cuts:lagn}.  The authors propose to work with $\Pi_{\omega}$ which is a span of the Benders' coefficients identified in the previous iterations. This choice of $\Pi_{\om}$ may not produce tight cuts, but the authors show that they help improve the bound at the root node. To achieve optimality, authors use Benders and Integer L-shaped cuts after adding Lagrangian cuts at the root node. 

{To test the performance of the restricted separation of Lagrangian cuts against the exact separation in alternating cut approach, we conduct an experiment with instances of stochastic server location problem (SSLP) from \citealt{ntaimo2005million}.} These instances are different from those considered in \citealt{chen2022generating} for which the LP relaxation of the subproblems have integer solutions up to high precision. This makes the exact evaluation of the subproblems and consequently, the separation problem computationally easier. We use the publicly available Python code provided in \cite{chen2022generating}. Further, for a fair comparison, we implement our alternating cut procedure in Python 3.9 with cuts added using lazy callbacks. We implement both the single-cut and multi-cut versions. Gurobi 11.0.2 is used as the optimization solver on an Apple M1 Pro machine with 32 GB memory. {Our python code is available in the Github repository \url{https://github.com/akulbansal5/twostage/tree/master} and \citealt{chen2022generating}'s code is available in the Github repository \url{https://github.com/rchen234/SIP_LAG}.}

The results are reported in Table \ref{tab:chen}. Some of the columns names are of the form \textbf{x}-\textbf{y}. Symbol $\textbf{x} \in \br{\text{single}, \text{multi}}$ denotes the results for single-cut and multi-cut versions, respectively. Symbol $y \in \{\text{benders},$ $\text{integer}, \text{time(s)} \}$, denotes the number of Benders' cuts, Integer L-shaped cuts, and solve time until termination with less than 0.01\% gap, respectively. The column names \textbf{C-LagTime(s)} and \textbf{C-time(s)} denote the solution times for adding the Lagrangian cuts and total time until convergence with less than 0.01\% gap using the algorithm in \citealt{chen2022generating}. {We consider different values for the hyper-parameters in the algorithm, such as gap tolerance and choice of normalization, and report the best results in Table \ref{tab:chen}.} We observe that both the multi-cut and single-cut versions of the alternating-cut approach with Integer L-shaped cuts outperform this method. The solve times with the single-cut version are noticeably better than those with the multi-cut version. Further, it is not immediately clear how to obtain a single-cut version from cuts of the form \eqref{cut:chen_lagrn}. In some cases, the obtained coefficients $\lambda_{0, \om} \approx 0$ in which case we can not aggregate the constraints as:   
\begin{align*}
    \theta \geq \sum_{\om} p_{\om} \frac{1}{\lambda_0,\om} \Big( \mc{L}_{\om}(\lambda_{\om}, \lambda_{0,\om}) - \lambda_{\om}^{\top}x \Big).
\end{align*}

Owing to these reasons, we do not implement the specialized Lagrangian cut-generation method in \citealt{chen2022generating} for the multi-stage setting.

\begin{table}[]
\centering
\caption{Comparison of alternating cut generation method implemented as single and multi-cut, with \citealt{chen2022generating}'s restricted separation of Lagrangian cuts.}
\label{tab:chen}
\resizebox{\textwidth}{!}{%
\begin{tabular}{c|ccc|ccc|cc}
\textbf{\textbf{Instance}} & \textbf{single-benders} & \textbf{single-integer} & \textbf{single-time (s)} & \textbf{\textbf{multi-benders}} & \textbf{\textbf{multi-integer}} & \textbf{\textbf{multi-time(s)}} & \textbf{\textbf{C-LagtTme(s)}} & \textbf{\textbf{C-time(s)}} \\ \hline
sslp\_5\_25\_50 & 21 & 3 & 1.18 & 1079 & 34 & 1.94 & 5.79 & 5.86 \\
sslp\_5\_25\_100 & 23 & 4 & 1.76 & 2151 & 36 & 1.74 & 13.10 & 13.27 \\
sslp\_10\_50\_50 & 219 & 5 & 13.49 & 6596 & 167 & 37.61 & 91.43 & 92.44 \\
sslp\_10\_50\_100 & 218 & 5 & 24.13 & 12854 & 162 & 53.28 & 171.20 & 173.20 \\
sslp\_10\_50\_500 & 233 & 14 & 152.07 & 121521 & 10938 & 670.82 & 776.80 & 839.06 \\
sslp\_10\_50\_1000 & 212 & 5 & 203.36 & 233328 & 20450 & 1281.95 & 1461.37 & 1423.55\\
sslp\_10\_50\_2000 & 201 & 4 & 374.11 & 448660 & 58687 & 2546.86 & 2777.85 & 2828.86 \\
sslp\_15\_45\_5 & 231 & 12 & 3.41 & 353 & 57 & 7.76 & 29.72 & 30.73 \\
sslp\_15\_45\_10 & 533 & 32 & 13.31 & 1660 & 278 & 36.59 & 48.64 & 52.21 \\
sslp\_15\_45\_15 & 588 & 46 & 27.19 & 832 & 135 & 25.06 & 86.01 & 86.94
\end{tabular}%
}
\end{table}

\bibliography{ref.bib}

\end{document}